\documentclass[11pt]{article}

\usepackage[english]{babel}
\usepackage{amsmath,amsfonts,amssymb,amsthm}
\usepackage{latexsym}
\usepackage{makeidx}
\usepackage{amscd}

\usepackage{vmargin}
\setmarginsrb{12mm}{10mm}{12mm}{20mm}%
            {10mm}{10mm}{10mm}{10mm}
\numberwithin{equation}{section}

\newtheorem{deff}{Definition}[section]
\newtheorem{lemma}[deff]{Lemma}
\newtheorem{theorem}[deff]{Theorem}
\newtheorem{corollary}[deff]{Corollary}

\newtheorem{proposition}[deff]{Proposition}
\newtheorem{fact}[deff]{Fact}
\newtheorem{em-example}[deff]{Example}
\newtheorem{em-def}[deff]{Definition}        
\newtheorem{em-remark}[deff]{Remark}         
\newtheorem{em-question}[deff]{Question}

\newenvironment{example}{\begin{em-example} \em }{ \end{em-example}}
\newenvironment{remark}{\begin{em-remark} \em }{\end{em-remark}}

\newcommand{\R}{\mathbb R}
\newcommand{\N}{\mathbb N}

\newcommand{\Q}{\mathbb Q}

\newcommand{\Z}{\mathbb Z}
\newcommand{\Prm}{\mathbb P}
\newcommand{\T}{\mathbb{T}}

\newcommand{\f}{\phi}
\newcommand{\CC}{\mathcal C}
\def\ent{\mathrm{ent}}
\def\aent{\mathrm{ent}^\star}

\def\Hom{\mathrm{Hom}}
\def\Ab{\mathbf{Ab}}
\def\TopAb{\mathbf{TopAb}}
\def\End{\mathrm{End}}

\global\def\hull#1{\langle{#1}\rangle}
\global\def\card#1{\left|{#1}\right|}

\global\def\Pdual#1{\widehat{#1}}
\global\def\prdual#1{{#1}^\dag}
\global\def\dual#1{{#1}'}

\input xy
\xyoption{all}

\title{Adjoint entropy vs Topological entropy}

\date{Dedicated to the sixtieth birthday of Dikran Dikranjan}

\author{Anna Giordano Bruno
\\{\footnotesize {\tt  anna.giordanobruno@uniud.it}} 
\\{\footnotesize Dipartimento di Matematica e Informatica,}
\\{\footnotesize Universit\`a di Udine,}
\\{\footnotesize Via delle Scienze, 206 - 33100 Udine}
}

\begin{document}

\maketitle


%\ead{anna.giordanobruno@uniud.it}
%\address{Dipartimento di Matematica e Informatica, Universit\`a di Udine, via delle Scienze, 206 - 33100 Udine}

\abstract{
Recently the adjoint algebraic entropy of endomorphisms of abelian groups was introduced and studied in \cite{DGS}. We generalize the notion of adjoint entropy to continuous endomorphisms of topological abelian groups. Indeed, the adjoint algebraic entropy is defined using the family of all finite-index subgroups, while we take only the subfamily of all \emph{open} finite-index subgroups to define the topological adjoint entropy. This allows us to compare the (topological) adjoint entropy with the known topological entropy of continuous endomorphisms of compact abelian groups. In particular, the topological adjoint entropy and the topological entropy coincide on continuous endomorphisms of totally disconnected compact abelian groups. Moreover, we prove two Bridge Theorems between the topological adjoint entropy and the algebraic entropy using respectively the Pontryagin duality and the precompact duality.
}

%\begin{keyword}
%Algebraic entropy \sep Adjoint entropy \sep Topological entropy \sep Pontryagin duality \sep Abelian groups

%\MSC primary 20K30 \sep
%secondary 28D20 \sep 22D35
%\end{keyword}

%\end{frontmatter}

\section{Introduction}

In the same paper \cite{AKM} where they introduced and studied the topological entropy $h_{top}$ for continuous maps of compact spaces (see Section \ref{aent-top-sec} for the precise definition), Adler, Konheim and McAndrew defined the algebraic entropy of endomorphisms of abelian groups. The notion of algebraic entropy was studied later by Weiss \cite{W}, who gave its basic properties and its relation with the topological entropy.

\smallskip
The algebraic entropy of an endomorphism $\phi$ of an abelian group $G$ measures to what extent $\phi$ moves the finite subgroups $F$ of $G$. More precisely, according to  \cite{AKM,W}, for a positive integer $n$, let $$T_n(\phi,F)=F+\phi(F)+\ldots+\phi^{n-1}(F)$$ be the \emph{$n$-th $\phi$-trajectory} of $F$. The \emph{algebraic entropy of $\phi$ with respect to $F$} is $$H(\phi,F)={\lim_{n\to \infty}\frac{\log|T_n(\phi,F)|}{n}},$$ and the \emph{algebraic entropy} of $\phi:G\to G$ is $$\ent(\phi)=\sup\{H(\phi,F): F\ \text{is a finite subgroup of } G\}.$$
Clearly $\ent(\phi)=\ent(\phi\restriction_{t(G)})$, where $t(G)$ denotes the torsion part of $G$; so the natural context for the algebraic entropy is that of endomorphisms of torsion abelian groups. Many papers in the last years were devoted to the study of various aspects of the algebraic entropy, and mainly \cite{DGSZ}, where its fundamental properties were proved.

\medskip
In analogy to the algebraic entropy, in \cite{DGS} the adjoint algebraic entropy of endomorphisms $\phi$ of abelian groups $G$ was introduced ``replacing" the family of all finite subgroups of $G$ with the family $\CC(G)$ of all finite-index subgroups of $G$. Indeed, for $N\in\CC(G)$, and a positive integer $n$, let $$B_n(\phi,N)=N\cap\phi^{-1}(N)\cap\ldots\cap\phi^{-n+1}(N);$$ the \emph{$n$-th $\phi$-cotrajectory} of $N$ is $$C_n(\phi,N)=\frac{G}{B_n(\phi,N)}.$$
Each $C_n(\phi,N)$ is finite, as each $B_n(\phi,N)\in\CC(G)$, because the family $\CC(G)$ is stable under inverse images with respect to $\phi$ and under finite intersections. The \emph{adjoint algebraic entropy of $\phi$ with respect to $N$} is 
\begin{equation}\label{H*}
H^\star(\phi,N)={\lim_{n\to \infty}\frac{\log|C_n(\phi,N)|}{n}}.
\end{equation}
This limit exists and it is finite;
in fact, as shown in {\cite[Proposition 2.3]{DGS}}, 
\begin{equation}\label{C=C_n->ent*=0}
\text{the sequence $\alpha_n=\card{\frac{C_{n+1}(\phi,N)}{C_n(\phi,N)}}=\card{\frac{B_n(\phi,N)}{B_{n+1}(\phi,N)}}$ is stationary;}
\end{equation}
more precisely, there exists a natural number $\alpha$ such that $\alpha_n= \alpha$ for all $n$ large enough. 
The \emph{adjoint algebraic entropy} of $\phi:G\to G$ is $$\aent(\phi)=\sup\{H^\star(\phi,N): N\in\CC(G)\}.$$  

\smallskip
While the algebraic entropy has values in $\log\N_+\cup\{\infty\}$, the adjoint algebraic entropy takes values only in $\{0,\infty\}$ as shown by \cite[Theorem 7.6]{DGS}.
This particular ``binary behavior'' of the values of the adjoint algebraic entropy seems to be caused by the fact that the family of finite-index subgroups can be very large. So in this paper we take only a part of it using the topology in the following way.
For a topological abelian group $(G,\tau)$, consider the subfamily $$\CC_\tau(G)=\{N\in\CC(G):N\ \text{$\tau$-open}\}$$ of $\CC(G)$ consisting of all $\tau$-open finite-index subgroups of $G$. 

\begin{deff}
For a topological abelian group $(G,\tau)$ and a continuous endomorphism $\phi:(G,\tau)\to (G,\tau)$, the 
\emph{topological adjoint entropy} of $\phi$ with respect to $\tau$ is
\begin{equation}\label{5}
\ent_{\tau}^\star(\phi)=\sup \{ H^\star(\phi,N):N\in\CC_\tau(G)\}.
\end{equation}
\end{deff}

Roughly speaking, $\aent_{\tau}$ is a variant of $\aent$ but taken only with respect to \emph{some} finite-index subgroups, namely, the $\tau$-open ones. Clearly, for $(G,\tau)$ a topological abelian group and $\phi:(G,\tau)\to(G,\tau)$ a continuous endomorphism, $\aent(\phi)\geq\aent_\tau(\phi)$.
For the discrete topology $\delta_G$ of $G$, we have $\aent_{\delta_G}(\phi)=\aent(\phi)$, so the notion of topological adjoint entropy extends that of adjoint algebraic entropy, and this provides a first motivation for this paper.

\medskip
In Section \ref{aent-sec} we study the basic properties of the topological adjoint entropy, with respect to the typical properties of the known entropies. 

Note that we defined the topological adjoint entropy for a \emph{continuous} endomorphism $\phi:(G,\tau)\to (G,\tau)$ of a topological abelian group $(G,\tau)$. From a categorical point of view this seems to be the right setting, also as $\CC_\tau(G)$ is stable under inverse images with respect to $\phi$ just because $\phi$ is continuous. On the other hand, the definition formally makes sense for not necessarily continuous endomorphisms, and furthermore the basic properties proved in Section \ref{aent-sec} hold true even removing the continuity of the endomorphisms. So it may be desirable to study also this more general situation.

\medskip
Recall that a topological abelian group $(G,\tau)$ is \emph{totally bounded} if for every open neighborhood $U$ of $0$ in $(G,\tau)$ there exists a finite subset $F$ of $G$ such that $U+F=G$. Moreover, $(G,\tau)$ is \emph{precompact} if it is Hausdorff and totally bounded. Since the completion $\widetilde{(G,\tau)}$ of a precompact abelian group $(G,\tau)$ is compact, the precompact abelian groups are precisely the subgroups of compact abelian groups.
Moreover, the group topology $\tau$ is said to be \emph{linear} if it has a base of the neighborhoods of $0$ consisting of open subgroups.

In Section \ref{ft-sec} we consider other properties of the topological adjoint entropy, giving topological and algebraic reductions for its computation. In particular, we see that in order to compute the topological adjoint entropy of continuous endomorphisms of topological abelian groups, it suffices to consider precompact linear topologies (see Remark \ref{wlog}). 

\medskip
In Section \ref{aent-top-sec}, which contains the main results of this paper, we study the connections of the topological adjoint entropy with the topological entropy and the algebraic entropy.

\smallskip
Note that the adjoint algebraic entropy is closely related to the algebraic entropy through the Pontryagin duality (see Section \ref{Pontryagin} for basic definitions and properties of the Pontryagin duality, see also \cite{HR,P}). Indeed, \cite[Theorem 5.3]{DGS}, which is a so-called Bridge Theorem, shows that: 

\begin{theorem}[Bridge Theorem $\aent$-$\ent$]\label{BT-DGS}
{The adjoint entropy of an endomorphism $\phi$ of an abelian group $G$ is the same as the algebraic entropy of the dual endomorphism $\Pdual\phi$ of $\phi$ of the (compact) Pontryagin dual $\Pdual G$ of $G$.}
\end{theorem}

Moreover, Weiss proved the Bridge Theorem \cite[Theorem 2.1]{W} relating the topological entropy and the algebraic entropy through the Pontryagin duality: 

\begin{theorem}[Bridge Theorem $h_{top}$-$\ent$]\label{BT-W}
{The topological entropy of a continuous endomorphism $\phi$ of a totally disconnected compact abelian group $(G,\tau)$ is the same as the algebraic entropy of the dual endomorphism $\Pdual\phi$ of $\phi$ of the (torsion and discrete) Pontryagin dual $\Pdual{(G,\tau)}$ of $(G,\tau)$.}
\end{theorem}

First of all we see that the topological adjoint entropy appeared in some non-explicit way in the proof of this classical theorem of Weiss.
More precisely, Theorem \ref{ent*=htop} shows that \emph{the topological adjoint entropy coincides with the topological entropy} for continuous endomorphisms of totally disconnected compact abelian groups (in particular, the topological adjoint entropy takes every value in $\log\N_+\cup\{\infty\}$, a more natural behavior with respect to that of the adjoint algebraic entropy).
Therefore, the topological adjoint entropy comes as an alternative form of topological entropy, with the advantage that it is defined for every continuous endomorphism of every abelian topological group, and not only in the compact case. So this provides a second motivation for introducing the adjoint topological entropy.

\smallskip
Here comes a third motivation for the topological adjoint entropy. Indeed, the following Bridge Theorem \ref{BT!-intro}, relating the topological adjoint entropy and the algebraic entropy through the Pontryagin duality, extends Weiss' Bridge Theorem \ref{BT-W} from totally disconnected compact abelian groups to arbitrary compact abelian groups. On the other hand, this is not possible using the topological entropy, so in some sense the adjoint topological entropy may be interpreted as the true ``dual entropy'' of the algebraic entropy. 

\begin{theorem}[Bridge Theorem $\aent_\tau$-$\ent$]\label{BT!-intro}
{The topological adjoint entropy of a continuous endomorphism $\phi$ of a compact abelian group $(G,\tau)$ is the same as the algebraic entropy of the dual endomorphism $\Pdual\phi$ of $\phi$ of the (discrete) Pontryagin dual $\Pdual{(G,\tau)}$ of $(G,\tau)$.}
\end{theorem}

In analogy to the Bridge Theorem \ref{BT-DGS} between the adjoint algebraic entropy and the algebraic entropy, we give also another Bridge Theorem, making use of the precompact duality (see Section \ref{precompact} for the basic definitions and properties of the precompact duality, see also \cite{MO,RT}):

\begin{theorem}[Bridge Theorem $\aent_\tau$-$\ent$]\label{ent-adj-top-intro}
{The topological adjoint entropy of a continuous endomorphism $\phi$ of a precompact abelian group $(G,\tau)$ is the same as the algebraic entropy of the dual endomorphism $\prdual\phi$ of $\phi$ of the precompact dual $\prdual{(G,\tau)}$ of $(G,\tau)$.}
\end{theorem}

The main examples in the context of entropy are the Bernoulli shifts; for $K$ an abelian group,
\begin{itemize}
\item[(a)] the \emph{right  Bernoulli shift} $\beta_K$ and the \emph{left Bernoulli shift} ${}_K\beta$ of the group $K^{\N}$ are defined respectively by 
$$\beta_K(x_0,x_1,x_2,\ldots)=(0,x_0,x_1,\ldots)\ \mbox{and}\ {}_K\beta(x_0,x_1,x_2,\ldots)=(x_1,x_2,x_3,\ldots);$$
\item[(b)] the \emph{two-sided Bernoulli shift} $\overline{\beta}_K$ of the group $K^{\Z}$ is defined by 
$$\overline\beta_K((x_n)_{n\in\Z})=(x_{n-1})_{n\in\Z}, \mbox{ for } (x_n)_{n\in\Z}\in K^{\Z}.$$
\end{itemize}
Since $K^{(\N)}$ is both $\beta_K$-invariant and ${}_K\beta$-invariant, and $K^{(\Z)}$ is $\overline\beta_K$-invariant, let $$\beta_K^\oplus=\beta_K\restriction_{K^{(\N)}},\ {}_K\beta^\oplus= {}_K\beta\restriction_{K^{(\N)}}\ \text{and}\ \overline\beta_K^\oplus=\overline\beta_K\restriction_{K^{(\Z)}}.$$
The adjoint algebraic entropy takes value infinity on the Bernoulli shifts, that is, if $K$ is a non-trivial abelian group, then 
$$\aent(\beta_K^\oplus)=\aent({}_K\beta^\oplus)=\aent(\overline\beta_K^\oplus)=\infty.$$
These values were calculated in \cite[Proposition 6.1]{DGS} applying \cite[Corollary 6.5]{AGB} and the Pontryagin duality (see Example \ref{beta}); we give here a direct computation in Proposition \ref{beta-d}. 
This result was one of the main steps in proving the above mentioned \cite[Theorem 7.6]{DGS} showing that the adjoint algebraic entropy takes values only in $\{0,\infty\}$.

\subsection*{Aknowledgements}

I would like to thank Professor Dikranjan for his very useful comments and suggestions.

\subsection*{Notation and Terminology}

We denote by $\mathbb Z$, $\mathbb N$, $\mathbb N_+$, $\Q$ and $\R$ respectively the integers, the natural numbers, the positive integers, the rationals and the reals. For $m\in\mathbb N_+$, we use $\mathbb Z(m)$ for the finite cyclic group of order $m$. Moreover, we consider $\mathbb T=\mathbb R/\mathbb Z$ endowed with its compact topology.

Let $G$ be an abelian group. 
For a set $X$ we denote by $G^{(X)}$ the direct sum $\bigoplus_X G$ of $|X|$ many copies of $G$. 
Moreover, $\End(G)$ is the ring of all endomorphisms of $G$. We denote by $0_G$ and $id_G$ respectively the endomorphism of $G$ which is identically $0$ and the identity endomorphism of $G$.

For a topological abelian group $(G,\tau)$, if $X$ is a subset of $G$, then $\overline X^\tau$ denotes the closure of $X$ in $(G,\tau)$; if there is no possibility of confusion we write simply $\overline X$. Moreover, if $H$ is a subgroup of $G$, we denote by $\tau_q$ the quotient topology of $\tau$ on $G/H$. If $(G,\tau)$ is Hausdorff, then $\widetilde{(G,\tau)}$ is the completion of $(G,\tau)$ and for $\phi:(G,\tau)\to (G,\tau)$ a continuous endomorphism, $\widetilde\phi:\widetilde{(G,\tau)}\to\widetilde{(G,\tau)}$ is the (unique) continuous extension of $\phi$ to $\widetilde{(G,\tau)}$. Furthermore, $c(G)$ denotes the connected component of $G$.

\section{Background on dualities}\label{Pontryagin}\label{precompact}

Let $(G,\tau)$ be an abelian group. The \emph{dual} group $\dual{(G,\tau)}$ of $(G,\tau)$ is the group of all continuous homomorphisms $(G,\tau)\to\mathbb T$, endowed with the discrete topology. If $\phi:(G,\tau)\to (G,\tau)$ is a continuous endomorphism, its dual endomorphism
\begin{center}
$\dual\phi:\dual{(G,\tau)}\to \dual{(G,\tau)}$ is defined by $\dual\phi(\chi)=\chi\circ\phi$ for every $\chi\in\dual{(G,\tau)}$. 
\end{center}
For $A\subseteq G$ and $B\subseteq \dual{(G,\tau)}$, the \emph{annihilator} $A^\bot$ of $A$ in $\dual{(G,\tau)}$ and the \emph{annihilator} $B^\top$ of $B$ in $G$ are respectively $$A^\bot=\{\chi\in\dual{(G,\tau)}:\chi(A)=0\}\ \text{and}\ B^\top=\{g\in G: \chi(g)=0,\forall \chi\in B\}.$$ 
Moreover, one can consider the map
$$\omega_G: (G,\tau)\to (G,\tau)''\ \text{defined by}\ \omega_G(g)(\chi)=\chi(g)\ \text{for every}\ g\in G\ \text{and}\ \chi\in\dual{(G,\tau)}.$$

\smallskip
For a locally compact abelian group $(G,\tau)$ the Pontryagin dual $\Pdual {(G,\tau)}$ of $(G,\tau)$ is $\dual{(G,\tau)}$ endowed with the compact-open topology \cite{P}, denoted here by $\Pdual\tau$. The Pontryagin dual of a locally compact abelian group is locally compact as well, and the Pontryagin dual of a (discrete) abelian group is always compact \cite{HR,P}. The map $\omega_G: (G,\tau)\to \widehat{\widehat{(G,\tau)}}$ is a topological isomorphism.
Moreover, for a continuos endomorphism $\phi:(G,\tau)\to (G,\tau)$, its dual endomorphism $\dual\phi:\Pdual{(G,\tau)}\to\Pdual{(G,\tau)}$ is continuous, and it is usually denoted by $$\Pdual\phi:\Pdual{(G,\tau)}\to\Pdual{(G,\tau)}.$$ 
For basic properties concerning the Pontryagin duality see \cite{DPS,HR}.

\medskip
We pass now to the precompact duality, studied mainly in \cite{MO,RT}. 

\smallskip
Assume that $(G,\tau)$ is precompact.
Since $(G,\tau)$ is dense in its compact completion $K=\widetilde{(G,\tau)}$, and since the continuous characters of $(G,\tau)$ are uniformly continuous, they can be extended to $K$, and it follows that $\dual{(G,\tau)}=\dual{K}$.

The \emph{precompact dual} $\prdual{(G,\tau)}$ of a precompact abelian group $(G,\tau)$ is $\dual{(G,\tau)}$ endowed with the weak topology $\prdual\tau$ generated by all elements of $G$ considered as continuous characters; this is the topology of $\dual{(G,\tau)}$ inherited from the product topology of $\mathbb T^G$, since $\dual{(G,\tau)}$ can be considered as a subgroup of $\mathbb T^G$. Then $\prdual{(G,\tau)}$ is a precompact abelian group.
As proved in \cite[Theorem 1]{RT}, the map $\omega_G: (G,\tau)\to (G,\tau)^{\dag\dag}$ is a topological isomorphism. The same result follows from \cite[Propositions 2.8, 3.9 and Theorem 3.11]{MO}. Moreover, if $\phi:(G,\tau)\to (G,\tau)$ is a continuous endomorphism, then the dual endomorphism $\dual\phi:\prdual{(G,\tau)}\to \prdual{(G,\tau)}$ is continuous \cite{RT}, and it is usually denoted by $$\prdual\phi:\prdual{(G,\tau)}\to \prdual{(G,\tau)}.$$

The precompact duality has the same basic properties of the Pontryagin duality. Indeed, also in this case the annihilators are closed subgroups of $\prdual{(G,\tau)}$ and $(G,\tau)$; furthermore, by \cite[Remark 10]{RT}, if $(G,\tau)$ is a precompact abelian group and $H$ a closed subgroup of $(G,\tau)$, then:
\begin{itemize}
\item[(a)] $\prdual{(G/H,\tau_q)}$ is topologically isomorphic to $(H^\bot,\prdual\tau\restriction_{H^\bot})$;
\item[(b)] $\prdual{(H,\tau\restriction_H)}$ is topologically isomorphic to $(\prdual{(G,\tau)}/H^\bot,(\prdual\tau)_q)$;
\item[(c)] the map $H\mapsto H^\bot$ defines an order-inverting bijection between the closed subgroups of $(G,\tau)$ and the closed subgroups of $\prdual{(G,\tau)}$.
\end{itemize}

Moreover, it is possible to prove the following properties, exactly as their counterparts for the Pontryagin duality (see \cite{HR} for (a) and (b), and \cite{DGS} for (c)).

\begin{fact}\label{prd}
\begin{itemize}
\item[(a)] If $F$ is a finite abelian group, then $\prdual F\cong F$.
\item[(b)] If $H_1,\ldots,H_n$ are subgroups of a precompact abelian group $G$, then $(\sum_{i=1}^nH_i)^\bot\cong_{top} \bigcap_{i=1}^nH_i^\bot$ and $(\bigcap_{i=1}^nH_i)^\bot\cong_{top}\overline{\sum_{i=1}^n H_i^\bot}$.
\item[(c)] If $G$ is a precompact abelian group, $H$ a subgroup of $G$ and $\phi:G\to G$ a continuous endomorphism, then $(\phi^{-n}(H))^\bot=(\prdual\phi)^n (H^\bot)$ for every $n\in\N$. 
\end{itemize}
\end{fact}

\section{Basic properties of the topological adjoint entropy}\label{aent-sec}

In this section we show the basic properties of the topological adjoint entropy, with respect to the usual basic properties possessed by the known entropies, as in particular the algebraic entropy and the topological entropy. These properties were discussed in \cite[Section 8]{DGS} for the adjoint algebraic entropy. Since our setting involving topology is more general, they require a new verification, even if formulas from \cite{DGS} are applied in the proofs.

\medskip
We start giving the following easy observation about the monotonicity of $H^\star(\phi,-)$ with respect to subgroups.

\begin{lemma}\label{N<M->HN>HM}
Let $G$ be an abelian group, $\phi\in\End(G)$ and $N,M\in\CC(G)$. If $N\subseteq M$, then $B_n(\phi,N)\subseteq B_n(\phi,M)$ and so $|C_n(\phi,N)|\geq|C_n(\phi,M)|$ for every $n\in\N_+$. Therefore, $H^\star(\phi,N)\geq H^\star(\phi,M)$. \hfill$\qed$
\end{lemma}

We have also the following monotonicity with respect to the topology.

\begin{lemma}\label{basic-tau-ent-b}
Let $G$ be an abelian group and $\tau,\tau'$ group topologies on $G$.
If $\tau\leq \tau'$ on $G$, then $\aent_{\tau}(\phi)\leq \aent_{\tau'}(\phi)$.
\end{lemma}

\begin{example}
It easily follows from the definition that $\aent_\tau(0_G)=\aent_\tau(id_G)=0$ for any topological abelian group $(G,\tau)$.
\end{example}

The first property is the invariance under conjugation.

\begin{lemma}\label{cbi}
Let $(G,\tau)$ be a topological abelian group and $\phi:(G,\tau)\to(G,\tau)$ a continuous endomorphism. If $(H,\sigma)$ is another topological abelian group and $\xi:(G,\tau)\to (H,\sigma)$ a topological isomorphism, then $$\aent_\sigma(\xi\circ\phi\circ\xi^{-1})=\aent_\tau(\phi).$$
\end{lemma}
\begin{proof}
Let $N\in\CC_\sigma(H)$ and call $\theta=\xi\circ\phi\circ\xi^{-1}$. Since $\xi$ is a topological isomorphism, $\xi^{-1}(N)\in\CC_\tau(G)$.
In the proof of \cite[Lemma 4.3]{DGS}, it is shown that
$H^\star(\theta,N)=H^\star(\phi,\xi^{-1}(N))$ for every $N\in\CC_\sigma(H)$, and hence $\aent_\sigma(\theta)=\aent_\tau(\phi)$.
\end{proof}

The second property is the so-called logarithmic law.

\begin{lemma}\label{ll}
Let $(G,\tau)$ be a topological abelian group and $\phi:(G,\tau)\to(G,\tau)$ a continuous endomorphism. Then for every $k\in\N_+$, $$\aent_\tau(\phi^k)=k\cdot \aent_\tau(\phi).$$
\end{lemma}
\begin{proof}
For $N\in\CC_\tau(G)$, fixed $k\in\N_+$, for every $n\in\N_+$ we have
$C_{nk}(\phi,N )=C_n(\phi^k,B_k(\phi,N))$.
Then, by the proof of \cite[Lemma 4.4]{DGS}, $$k\cdot H^\star(\phi,N)=H^\star(\phi^k,B_k(\phi,N))\leq\aent_\tau(\phi^k).$$
Consequently, $k\cdot \aent_\tau(\phi)\leq\aent_\tau(\phi^k)$.

Now we prove the converse inequality. Indeed, by the proof of \cite[Lemma 4.4]{DGS}, for $N\in\CC_\tau(G)$ and for $k\in\N_+$,
$$\aent_\tau(\phi)\geq H^\star(\phi,N)\geq \frac{H^\star(\phi^k,N)}{k}.$$
This shows that $k\cdot \aent_\tau(\phi)\geq \aent_\tau(\phi^k)$, that concludes the proof.
\end{proof}

The next lemma shows that a topological automorphism has the same topological adjoint entropy as its inverse.

\begin{lemma}
Let $(G,\tau)$ be a topological abelian group and $\phi:(G,\tau)\to(G,\tau)$ a topological automorphism. Then $\aent_\tau(\phi)=\aent_\tau(\phi^{-1})$.
\end{lemma}
\begin{proof}
For every $n\in\N_+$ and every $N\in\CC_\tau(G)$, we have $H^\star(\phi,N)=H^\star(\phi^{-1},N)$ by the proof of \cite[Lemma 4.5]{DGS}, and hence $\aent_\tau(\phi)=\aent_\tau(\phi^{-1})$.
\end{proof}

The following corollary is a direct consequence of the previous two results.

\begin{corollary}
Let $(G,\tau)$ be a topological abelian group and $\phi:(G,\tau)\to(G,\tau)$ a topological automorphism. Then $\aent_\tau(\phi^k)=|k|\cdot \aent_\tau(\phi)$ for every $k\in\Z$. 
\end{corollary}

The next property, a monotonicity law for induced endomorphisms on quotients over invariant subgroups, will be often used in the sequel.

\begin{lemma}\label{quotient}
Let $(G,\tau)$ be a topological abelian group, $\phi:(G,\tau)\to(G,\tau)$ a continuous endomorphism and $H$ a $\phi$-invariant subgroup of $G$. Then $\aent_\tau(\phi)\geq \aent_{\tau_q}(\overline\phi)$, where $\overline \phi:(G/H,\tau_q)\to (G/H,\tau_q)$ is the continuous endomorphism induced by $\phi$. 
\end{lemma}
\begin{proof}
Let $N/H\in\CC_{\tau_q}(G/H)$; then $N\in\CC_\tau(G)$. 
Since $H\subseteq N$ and $H$ is $\phi$-invariant, $H\subseteq \phi^{-n}(N)$ for every $n\in\N$. Consequently, $H\subseteq B_n(\phi,N)$ for every $n\in\N_+$.
Since $\overline\phi^{-n}(N/H)=\phi^{-n}(N)/H$ for every $n\in\N$, we have $B_n(\overline\phi,N/H)=B_n(\phi,N)/H$ for every $n\in\N_+$. Therefore, for every $n\in\N_+$,
$$C_n(\overline\phi,N/H)=(G/H)/(B_n(\phi,N)/H)\cong G/B_n(\phi,N)=C_n(\phi,N).$$ Hence,
\begin{equation}\label{Hq}
H^\star(\overline\phi,N/H)=H^\star(\phi,N),
\end{equation}
 and this proves $\ent^\star(\overline\phi)\leq\ent^\star(\phi)$.
\end{proof}

In general the topological adjoint entropy fails to be monotone with respect to restrictions to invariant subgroups \cite{DGS}.
Nevertheless, if we impose sufficiently stronger conditions on the invariant subgroup, we obtain more than the searched monotonicity in Lemma \ref{subgroup} and Proposition \ref{aenttilde.}

\begin{lemma}\label{subgroup}
Let $(G,\tau)$ be a topological abelian group, $\phi:(G,\tau)\to(G,\tau)$ a continuous endomorphism and $H$ a $\phi$-invariant subgroup of $(G,\tau)$. If $H\in\CC_\tau(G)$, then $\aent_\tau(\phi)=\aent_{\tau\restriction_H}(\phi\restriction_H)$.
\end{lemma}
\begin{proof}
Let $N\in\CC_{\tau\restriction_H}(H)$. Since $H\in\CC_\tau(G)$, it follows that $N\in\CC_\tau(G)$ as well.
By the proof of \cite[Lemma 4.9]{DGS} this implies $\aent_\tau(\phi)\geq\aent_{\tau\restriction_H}(\phi\restriction_H)$.
On the other hand, if $N\in\CC_\tau(G)$, then $N\cap H\in\CC_{\tau\restriction_H}(H)$, and $B_n(\phi,N)\cap H=B_n(\phi\restriction_H,N\cap H)$ for every $n\in\N_+$. Therefore, $H^\star(\phi,N)=H^\star(\phi\restriction_{H},N)$ and we can conclude that $\aent_\tau(\phi)\leq\aent_{\tau\restriction_H}(\phi\restriction_H)$. Hence, $\aent_\tau(\phi)\leq\aent_{\tau\restriction_H}(\phi\restriction_H)$.
\end{proof}

Let $G$ be an abelian group and $H$ a subgroup of $H$. Since there exists an injective homomorphism $\iota:H/N\cap H\to G/N$ induced by the inclusion $H\hookrightarrow G$, the map $\xi:\CC(G)\to \CC(H)$ defined by $N\mapsto N\cap H$ is well-defined.
Consider now the group topology $\tau$ on $G$. Then $\xi$ restricts to $\xi:\CC_{\tau}(G)\to\CC_{\tau\restriction_H}(H)$, and $\iota:H/N\cap H\to G/N$ is continuous with respect to the quotients topologies.
Moreover, we have the following

\begin{lemma}\label{xieta}
Let $(G,\tau)$ be a topological abelian group and $H$ a dense subgroup of $(G,\tau)$.
Then $\xi:\CC_\tau(G)\to \CC_{\tau\restriction_H}(H)$ defined by $N\mapsto N\cap H$ is a bijection and its inverse is $\eta:\CC_{\tau\restriction_H}(H)\to \CC_{\tau}(G)$ defined by $M\mapsto \overline M$.
\end{lemma}
\begin{proof}
Let $N\in\CC_\tau(G)$. Then $M=N\cap H\in\CC_{\tau\restriction_H}(H)$ and $N=\overline M$ in $(G,\tau)$; indeed, $N\subseteq \overline{N\cap H}$, because $H$ is dense in $(G,\tau)$, and so $N=\overline{N\cap H}$, since $N$ is an open subgroup and so closed.
\end{proof}

\begin{proposition}\label{aenttilde.}
Let $(G,\tau)$ be a topological abelian group, $\phi:(G,\tau)\to(G,\tau)$ a continuous endomorphism and $H$ a dense $\phi$-invariant subgroup of $(G,\tau)$.
Then $\aent_\tau(\phi)=\aent_{\tau\restriction_H}(\phi\restriction_H).$
\end{proposition}
\begin{proof}
Let $N\in\CC_\tau(G)$ and $M=N\cap H$.
By Lemma \ref{xieta}, the continuous injective homomorphism $\iota:H/M\to G/N$ is also open. Since $H$ is dense in $(G,\tau)$, the image of $H/M$ under $\iota$ is dense in $G/N$, which is finite; hence $H/M\cong G/N$. Moreover, $H/B_n(\phi\restriction_H,M)\cong G/B_n(\phi,N)$ for every $n\in\N_+$; indeed, for every $n\in\N_+$, $$B_n(\phi,N)\cap H=\bigcap_{i=0}^{n-1}\phi^{-i}(N)\cap H=\bigcap_{i=0}^\infty(\phi^{-i}(N)\cap H)=\bigcap_{i=0}^{n-1}\phi^{-i}(N)=B_n(\phi,M).$$ Applying the definition, we have $H_{\tau}^\star(\phi,N)=H_\tau^\star(\phi\restriction_H,M)$ and so $\ent_{\tau}^\star(\phi)=\aent_{\tau\restriction_H}(\phi\restriction_H)$.
\end{proof}

\begin{corollary}\label{aenttilde}
Let $(G,\tau)$ be a Hausdorff abelian group, $\phi:(G,\tau)\to(G,\tau)$ a continuous endomorphism, and denote by $\widetilde\tau$ the topology of the completion $\widetilde{(G,\tau)}$ of $(G,\tau)$ and by $\widetilde\phi:\widetilde{(G,\tau)}\to\widetilde{(G,\tau)}$ the extension of $\phi$.
Then $$\aent_\tau(\phi)=\aent_{\widetilde\tau}(\widetilde\phi).$$
\end{corollary}

Now we verify the additivity of the topological adjoint entropy for the direct product of two continuous endomorphisms.

\begin{proposition}
Let $(G,\tau)$ be a topological abelian group.
If $(G,\tau) = (G_1,\tau_1) \times (G_2,\tau_2)$ for some subgroups $G_1,G_2$ of $G$ with $\tau_1=\tau\restriction_{G_1},\tau_2=\tau\restriction_{G_2}$, and $\f= \f_1 \times \f_2: G \to G$ for some continuous $\f_1:(G,\tau_1)\to (G,\tau_1), \f_2:(G_2,\tau_2)\to (G_2,\tau_2)$, then $$\aent_\tau(\f) = \aent_{\tau_1}(\f_1) + \aent_{\tau_2}(\f_2).$$   
\end{proposition}
\begin{proof}
Let $N\in\CC_\tau(G)$. Then $N$ contains a subgroup of the form $N'=N_1\times N_2$, where $N_i=N\cap G_i\in \CC_{\tau_i}(G_i)$ for $i=1,2$; in particular, $N'\in\CC_\tau(G)$.
For every $n\in\N_+$, we have $$|C_n(\phi,N)|\leq |C_n(\f, N')|=|C_n(\f_1,N_1)\times C_n(\f_2,N_2)|$$ and so $$H^\star(\phi,N)\leq H^\star(\phi,N')=H^\star(\phi_1,N_1)+H^\star(\phi_2,N_2).$$
Now the thesis follows from the definition.
\end{proof}

The continuity for inverse limits fails in general, as shown in \cite{DGS}. But it holds in the particular case of continuous endomorphisms $\phi$ of totally disconnected compact abelian groups $(K,\tau)$. In fact, Theorem \ref{ent*=htop} will show that in this case the topological adjoint entropy of $\phi$ coincides with the topological entropy of $\phi$, and it is known that the topological entropy is continuous for inverse limits \cite{AKM}.

\section{Topological adjoint entropy and functorial topologies}\label{ft-sec}

Let $(G,\tau)$ be a topological abelian group and $\phi:(G,\tau)\to (G,\tau)$ a continuous endomorphism.
Obviously, $\aent_{\tau}(\phi)\leq \aent(\phi)$ since the supremum in the definition of the topological adjoint entropy (see \eqref{5}) is taken over the subfamily $\CC_\tau(G)$ of $\CC(G)$ consisting only of the $\tau$-open finite-index subgroups of $G$. 

For an abelian group $G$, the \emph{profinite topology} $\gamma_G$ of $G$ has $\CC(G)$ as a base of the neighborhoods of $0$. So it is worth noting that $\CC(G)=\CC_{\gamma_G}(G)$, and the next properties easily follow:

\begin{lemma}
Let $(G,\tau)$ be a topological abelian group and $\phi:(G,\tau)\to (G,\tau)$ a continuous endomorphism. Then:
\begin{itemize}
\item[(a)] $\aent(\phi)=\aent_{\gamma_G}(\phi)$;
\item[(b)] if $\tau\geq\gamma_G$, then $\aent_\tau(\phi)=\aent_{\gamma_G}(\phi)=\aent(\phi)$.
\end{itemize}
\end{lemma}

We consider now two modifications of a group topology, which are functors $\TopAb\to \TopAb$, where $\TopAb$ is the category of all topological abelian groups and their continuous homomorphisms.

If $(G,\tau)$ is a topological abelian group, the \emph{Bohr modification} of $\tau$ is the topology $$\tau^+=\sup\{\tau':\tau'\leq\tau,\tau' \text{totally bounded}\};$$ this topology is the finest totally bounded group topology on $G$ coarser than $\tau$. Actually, $\tau^+=\inf\{\tau,\mathcal P_G\}$, where $\mathcal P_G$ is the \emph{Bohr topology}, that is, the group topology of $G$ generated by all characters of $G$. Note that $\delta_G^+=\gamma_G$, where $\delta_G$ denotes the discrete topology of $G$.

\smallskip
Let $\Ab$ be the category of all abelian groups and their homomorphisms. Following \cite{F}, a \emph{functorial topology} is a class $\tau=\{\tau_A:A\in\Ab\}$, where $(A,\tau_A)$ is a topological group for every $A\in\Ab$, and every homomorphism in $\Ab$ is continuous. In other words, a functorial topology is a functor $\tau:\Ab\to\TopAb$ such that $\tau(A)=(A,\tau_A)$ for every $A\in\Ab$, where $\tau_A$ denotes the topology on $A$, and $\tau(\phi)=\phi$ for every morphism $\phi$ in $\Ab$ \cite{BM}.

For an abelian group $G$ the profinite topology $\gamma_G$ and the Bohr topology $\mathcal P_G$ are functorial topologies, as well as the \emph{natural topology} $\nu_G$, which has $\{mG:m\in\N_+\}$ as a base of the neighborhoods of $0$. Moreover, the profinite and the natural topology are linear topologies. These three functorial topologies are related by the following equality proved in \cite{DG2}:
\begin{equation}\label{pnb}
\gamma_G=\inf\{\nu_G,\mathcal P_G\}.
\end{equation}

\medskip
The second modification that we consider is the \emph{linear modification} $\tau_\lambda$ of $\tau$, that is, the group topology on $G$ which has all the $\tau$-open subgroups as a base of the neighborhoods of $0$.

\begin{lemma}\label{bl=p}
For $G$ an abelian group, $(\mathcal P_G)_\lambda=\gamma_G$.
\end{lemma}
\begin{proof}
Since $\gamma_G\leq\mathcal P_G$, we have $(\gamma_G)_\lambda\leq (\mathcal P_G)_\lambda$. Moreover, $\gamma_G=(\gamma_G)_\lambda$ and $(\mathcal P_G)_\lambda\leq\gamma_G$.
\end{proof}

\begin{proposition}\label{inf}
Let $(G,\tau)$ and $(H,\sigma)$ be abelian topological groups and $\phi:(G,\tau)\to (H,\sigma)$ a continuous surjective homomorphism.
\begin{itemize}
\item[(a)] If $\tau$ is linear and totally bounded, then $\sigma$ is linear and totally bounded as well.
\item[(b)] If $(G,\tau)$ is bounded torsion and totally bounded, then $\tau$ is linear.
\end{itemize}
\end{proposition}
\begin{proof}
(a) It is clear that $\sigma$ is totally bounded.

Assume first that $\tau$ is precompact. 
Since $\tau$ is linear and precompact, $\widetilde{(G,\tau)}$ is compact and linear. The extension $\widetilde\phi:\widetilde{(G,\tau)}\to\widetilde{(H,\sigma)}$ is continuous and surjective, so open as $\widetilde{(G,\tau)}$ is compact. Consequently, $\widetilde{(H,\sigma)}$ is compact and linear. Now we can conclude that $\sigma$ is linear as well.

If now $\tau$ is only totally bounded, consider the quotient $(G/\overline{\{0\}}^\tau,\tau_q)$ of $(G,\tau)$. Then $\tau_q$ is precompact, and it is linear since $\tau$ is linear. The homomorphism $\overline\phi:(G/\overline{\{0\}}^\tau,\tau_q)\to (G/\overline{\{0\}}^\sigma,\sigma_q)$ induced by $\phi$ is surjective and continuous. By the previous case of the proof we have that $\sigma_q$ is linear. Since $\sigma$ is the initial topology of $\sigma_q$ by the canonical projection, we can conclude that $\sigma$ is linear as well.

\smallskip
(b) Assume first that $\tau$ is precompact. Since $(G,\tau)$ is bounded torsion, its (compact) completion $\widetilde{(G,\tau)}$ is bounded torsion as well. Therefore, $\widetilde{(G,\tau)}$ is compact and totally disconnected, so linear.

If now $\tau$ is only bounded torsion, consider the quotient $(G/\overline{\{0\}}^\tau,\tau_q)$ of $(G,\tau)$. This quotient is precompact and so $\tau_q$ is linear by the previous part of the proof. Since $\tau$ is the initial topology of $\tau_q$ by the canonical projection, $\tau$ is linear as well.
\end{proof}

We collect in Lemma \ref{l+} some basic properties of the combination of the linear and the Bohr modifications. 

\begin{lemma}\label{l+}
Let $(G,\tau)$ be a topological abelian group. Then:
\begin{itemize}
\item[(a)] $(\tau_\lambda)^+=(\tau^+)_\lambda$, so we can write simply $\tau_\lambda^+$;
\item[(b)] $\tau_{\lambda}^+$ is linear and totally bounded;
\item[(c)] $\tau$ is linear and totally bounded if and only if $\tau=\tau_\lambda^+$;
\item[(d)] $\CC_{\tau_\lambda^+}(G)=\CC_{\tau_\lambda}(G)=\CC_{\tau^+}(G)=\CC_\tau(G)\subseteq \CC(G)$;
\item[(e)] $\CC_\tau(G)$ is a base of the heighbourhoods of $0$ in $(G,\tau_\lambda^+)$; 
\item[(f)] $\tau_\lambda^+=\inf\{\tau,\gamma_G\}$.
\end{itemize}
\end{lemma}
\begin{proof}
(a), (b), (c), (d) and (e) are clear.

\smallskip
(f) Clearly, $\tau_\lambda^+\leq\inf\{\tau,\gamma_G\}$. On the other hand, Proposition \ref{inf}(a) implies that $\inf\{\tau,\gamma_G\}$ is linear and totally bounded, hence it coincides with $\tau_\lambda^+$.
\end{proof}

In particular, we have the following diagram in the lattice of all group topologies of an abelian group $G$:
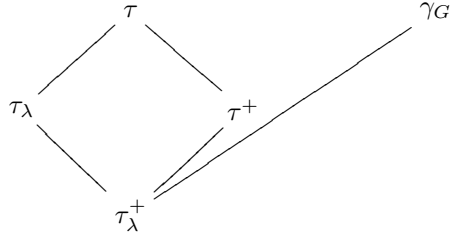
$$\xymatrix{
& \tau \ar@{-}[dl] \ar@{-}[dr]& & & \gamma_G \ar@{-}[ddlll]\\
\tau_\lambda \ar@{-}[dr] & & \tau^+ \ar@{-}[dl] & & \\
& \tau_\lambda^+ & & &
}$$

For an abelian group $G$, recall that $G^1=\bigcap_{n\in\N_+}n G$ is the \emph{first Ulm subgroup}, which is fully invariant in $G$.  It is well known that $G^1=\bigcap_{N\in\CC(G)}N$ \cite{F}.

For a topological abelian group $(G,\tau)$, in analogy to the first Ulm subgroup, define $$G^1_\tau=\bigcap_{N\in C_\tau(G)}N,$$
and let $\phi^1_\tau:(G/G_\tau^1,\tau_q)\to (G/G_\tau^1,\tau_q)$ be the continuous endomorphism induced by $\phi$, where $\tau_q$ is the quotient topology of $\tau$.
Note that, unlike the Ulm subgroup, its topological version $G^1_\tau$ may fail to coincide with $\bigcap\{nG:n\in\N_+,nG\in\CC_\tau(G)\}$ as the following example shows.

\begin{example}
Let $G=\Z(p)^{(\N)}$ for a prime $p$, and let $\tau$ be the product topology on $G$. Then $G^1_\tau=0$, while $\bigcap\{nG:n\in\N_+,nG\in\CC_\tau(G)\}=G$.
\end{example}

The equality proved in Proposition \ref{profin}(a) below shows the correct counterpart of the equality in the topological case.

\medskip
Since $G$ is residually finite if and only if $G^1=0$, we say that $G$ is \emph{$\tau$-residually finite} if $G^1_\tau=0$. Note that $(G/G^1_\tau)^1_{\tau_q}=0$.
Clearly, for $G$ an abelian group, $\gamma_G$-residually finite means residually finite. Furthermore, $G^1_\tau\supseteq G^1$ and so $\tau$-residually finite implies residually finite.
The following example shows that the converse implication does not hold true in general.

\begin{example}
Let $\alpha\in \mathbb T$ be an element of infinite order, consider the inclusion $\Z\to \T$ given by $1\mapsto \alpha$ and endow $\Z$ with the topology $\tau_\alpha$ inherited from $\T$ by this inclusion. Then $\Z^1=0$, while $\Z^1_{\tau_\alpha}=\Z$. In particular, $\Z$ is residually finite but not $\tau_\alpha$-residually finite.
\end{example}

Item (a) of Proposition \ref{profin} gives a different description of $G^1_\tau$, while item (b) shows the relation between the dual of a topological abelian group $(G,\tau)$ and the dual of $(G,\tau_\lambda^+)$; both parts use $t((G,\tau)')$.

\begin{proposition}\label{profin}
Let $(G,\tau)$ be a topological abelian group. 
\begin{itemize}
\item[(a)] Then $G^1_\tau=\bigcap_{\chi\in t((G,\tau)')}\ker\chi$.
\item[(b)] If $G$ is $\tau$-residually finite, then $(G,\tau_\lambda^+)'=t((G,\tau)')$.
\end{itemize}
\end{proposition}
\begin{proof}
Let $N\in\CC_\tau(G)$ and let $\chi\in (G,\tau)'$ be such that $\chi(N)=0$. Then $\chi\in t((G,\tau)')$. In fact, since $G/N$ is finite, there exists $m\in\N_+$ such that $mG\subseteq N$. Then $\chi(mG)=0$ and so $m\chi=0$; in particular, $\chi\in t((G,\tau)')$.

\smallskip
(a) Let $N\in\CC_\tau(G)$. Since $G/N$ is finite and discrete, $G/N$ is isomorphic to $\Z(k_1)\times\ldots\times\Z(k_n)$, for some $k_1,\ldots,k_n\in\N_+$. For each $i=1,\ldots,n$, let $\overline\chi_i:\Z(k_i)\to \T$ be the embedding. Then $\overline\chi=\chi_1\times\ldots\times\chi_n$ gives an embedding $G/N\to \T^n$. Let now $\chi=\overline\chi\circ\pi\in (G,\tau)'$, where $\pi:G\to G/N$ is the canonical projection. So $N=\ker\chi$. By the starting observation in the proof, $\chi\in t((G,\tau)')$.
This proves that $G^1_\tau\supseteq \bigcap_{\chi\in t((G,\tau)')}\ker\chi$. To verify the converse inclusion it suffices to note that, if $\chi\in t((G,\tau)')$, then $\ker\chi\in\CC_\tau(G)$.

\smallskip
(b) Fix a neighborhood $U$ of 0 in $\T$ that contains no non-zero subgroups of $\T$. For every continuous character $\chi:(G,\tau_\lambda^+)\to\T$, there exists $N\in\CC_\tau(G)$ such that $\chi(N)\subseteq U$. By the choice of $U$ this yields $\chi(N) = 0$, so $\chi\in t((G,\tau)'$ by the starting observation in the proof. Hence, we have verified the inclusion $(G,\tau_\lambda^+)'\subseteq t((G,\tau)')$. 

To prove the converse inclusion, let $\chi\in t((G,\tau)')$. We have to verify that $\chi:(G,\tau_\lambda^+)\to\T$ is continuous. First note that $\chi:(G,\tau^+)\to \T$ is continuous. Moreover, by hypothesis there exists $m\in\N_+$ such that $m\chi=0$, and so $\chi$ factorizes through the canonical projection $\pi:(G,\tau^+)\to (G/\overline{mG}^{\tau^+},\tau^+_q)$, where $\tau^+_q$ is the quotient topology of $\tau^+$, and the continuous character $\overline\chi:(G/\overline{mG}^{\tau^+},\tau^+_q)\to \T$, that is, $\chi=\overline\chi\circ\pi$. Now $(G/\overline{mG}^{\tau^+},\tau^+_q)$ is precompact and bounded torsion, so linear by Proposition \ref{inf}(b). This yields that $\chi:(G,\tau_\lambda^+)\to\T$ is continuous, hence $t((G,\tau)')\subseteq (G,\tau_\lambda^+)'$, and this concludes the proof.
\end{proof}

A consequence of this proposition is that, in case $G$ is a residually finite abelian group, then $(G,\gamma_G)'=t(\Hom(G,\T))$; this equality is contained in \cite[Lemma 3.2]{DG2}.
Since $\prdual \phi=\Pdual\phi\restriction_{t(\Pdual G)}$, Theorem \ref{BT-DGS} yields $\aent(\phi)=\ent(\widehat\phi\restriction_{t(\widehat G)})=\ent(\prdual\phi)$, where $\prdual\phi:\prdual{(G,\gamma_G)}\to\prdual{(G,\gamma_G)}$. In particular, $\aent(\phi)=\ent(\prdual\phi)$. The latter equality will be generalized by Theorem \ref{ent-adj-top}.

\medskip
The next lemma gives a characterization of $\tau$-residually finite abelian groups in terms of the modification $\tau_\lambda^+$ of $\tau$.

\begin{lemma}\label{res-fin-tau}
For a topological abelian group $(G,\tau)$ the following conditions are equivalent:
\begin{itemize}
\item[(a)] $G$ is $\tau$-residually finite;
\item[(b)] $G$ is $\tau_\lambda^+$-residually finite;
\item[(c)] $\tau_\lambda^+$ is Hausdorff (so precompact).
\end{itemize}
In particular, $G/G^1_\tau$ is $\tau_q$-residually finite, where $\tau_q$ is the quotient topology induced by $\tau$ on $G/G^1_\tau$.
\end{lemma}
\begin{proof}
(a)$\Leftrightarrow$(b) Since $\CC_\tau(G)=\CC_{\tau_\lambda^+}(G)$, we have that $G^1_\tau=G^1_{\tau_\lambda^+}$.

\smallskip
(a)$\Leftrightarrow$(c) It suffices to note that $G^1_\tau=\overline{\{0\}}^{\tau_\lambda^+}$.
\end{proof}

\begin{remark}
Let $(G,\tau)$ be a topological abelian group and let $\phi:(G,\tau)\to (G,\tau)$ be a continuous endomorphism.
\begin{itemize}
\item[(a)] To understand where the topology $\tau_\lambda^+$ comes from, consider the Bohr compactification $b(G,\tau)$ of $(G,\tau)$ with the canonical continuous endomorphism $\varrho_G:(G,\tau)\to b(G,\tau)$. 
We have the following diagram:
\begin{equation*}
\xymatrix{
&  b(G,\tau)\ar[dd]^{\pi} \\
(G,\tau)\ar[ru]^{\varrho_G} \ar[rd]^s \ar[dd]^{id_G}  & \\
& b(G,\tau)/c(b(G,\tau))\\
(G,\tau_\lambda^+) \ar[ru]^{s_\lambda} & 
}
\end{equation*}
where $\pi:b(G,\tau)\to b(G,\tau)/c(b(G,\tau))$ is the canonical projection, $s=\pi\circ\varrho_G$ and $s_\lambda$ is $s$ considered on $(G,\tau_\lambda^+)$, that is, $s=s_\lambda\circ id_G$. Then $\tau_\lambda^+$ is the initial topology of the topology of $b(G,\tau)/c(b(G,\tau))$, which is the (compact and totally disconnected) quotient topology of that of $b(G,\tau)$.
\item[(b)] The von Neumann kernel $\ker\varrho_G$ is contained in $G^1_\tau$. Indeed, for $N\in\CC_{\tau}(G)$, the quotient $G/N$ is finite and discrete. So the characters separate the points of $G/N$, and therefore $\ker\varrho_G\subseteq N$. Hence, $\ker\varrho_G\subseteq G^1_\tau$.
\item[(c)] Assume that $G$ is $\tau$-residually finite. We show that $s$ (and so $s_\lambda$) is injective.
Then $\varrho_G$ is injective by item (b). Moreover, $s$ is injective as well. In fact, $\ker\varrho_G=0$, and so we have $\ker s=G\cap c(b(G,\tau))$. Since $b(G,\tau)$ is compact, $c(b(G,\tau))=G^1_{\tau_b}$. Therefore, denoted by $\tau_b$ the topology of $b(G,\tau)$, $\ker s=G\cap \bigcap_{N\in\CC_{\tau_b}(b(G,\tau))}N=\bigcap_{N\in\CC_{\tau_b}(b(G,\tau))}(G\cap N)=\bigcap_{M\in\CC_\tau(G)}M=G^1_\tau$, and $G^1_\tau=0$ by hypothesis. Hence, $\ker s=0$.
\item[(d)] If $(G,\tau)$ is precompact and $G$ is $\tau$-residually finite, then $b(G,\tau)=\widetilde{(G,\tau)}$ and $b(G,\tau)/c(b(G,\tau))=\widetilde{(G,\tau_\lambda)}$. So the diagram in item (a) becomes:
\begin{equation}\label{leftside}
\xymatrix{
&  \widetilde{(G,\tau)}\ar[dd]^{\pi} \\
(G,\tau)\ar@{^{(}->}[ru] \ar[dd]^{id_G}  & \\
& \widetilde{(G,\tau_\lambda)}=\widetilde{(G,\tau)}/c(\widetilde{(G,\tau)})\\
(G,\tau_\lambda^+) \ar@{^{(}->}[ru] & 
}
\end{equation}
where $\pi=\widetilde{id}_G$. Note that the right hand side of the diagram consists of the completions of the groups on the left hand side.
\end{itemize}
\end{remark}

Now we pass to the main part of this section, in which we give reductions for the computation of the topological adjoint entropy.

\smallskip
The next proposition gives a first topological reduction for the computation of the topological adjoint entropy. Indeed, it shows that it is sufficient to consider group topologies which are totally bounded and linear. Note that if $\phi:(G,\tau)\to(G,\tau)$ is a continuous endomorphism of a topological abelian group $(G,\tau)$, then  $\phi:(G,\tau_\lambda^+)\to(G,\tau_\lambda^+)$ is continuous as well, since, as noted above, both the Bohr modification and the linear modification are functors $\mathbf{TopAb}\to \mathbf{Top Ab}$.

\begin{proposition}\label{lin+}
Let $(G,\tau)$ be a topological abelian group and $\phi:(G,\tau)\to (G,\tau)$ a continuous endomorphism. Then $$\aent_\tau(\phi)=\aent_{\tau_\lambda}(\phi)=\aent_{\tau^+}(\phi)=\aent_{\tau_\lambda^+}(\phi).$$
\end{proposition}
\begin{proof}
Since $\CC_\tau(G)=\CC_{\tau_\lambda^+}(G)$ by Lemma \ref{l+}(d), by the definition of topological adjoint entropy it is possible to conclude that $\aent_\tau(\phi)=\aent_{\tau_\lambda^+}(\phi)$. To prove the other two equalities, it suffices to apply Lemma \ref{basic-tau-ent-b}.
\end{proof}

The following result, which generalizes \cite[Proposition 4.13]{DGS}, is another reduction for the computation of the topological adjoint entropy. In fact, it shows that it is possible to restrict to $\tau$-residually finite abelian groups.

\begin{proposition}\label{ent*=ent*1-tau}
Let $(G,\tau)$ be a topological abelian group and $\phi:(G,\tau)\to (G,\tau)$ a continuous endomorphism. Then $\aent_\tau(\phi)=\aent_{\tau_q}(\phi_\tau^1)$, where $\tau_q$ is the quotient topology induced by $\tau$ on $G/G^1_\tau$.
\end{proposition}
\begin{proof}
By Lemma \ref{quotient}, $\aent_\tau(\phi)\geq\aent_{\tau_q}(\phi^1_\tau)$. Let $N\in\CC_\tau(G)$. Then $G^1_\tau\subseteq N$, and so $N/G^1_\tau\in\CC_{\tau_q}(G/G^1_\tau)$. By \eqref{Hq}, we know that $H^\star(\phi,N)=H^\star(\phi^1_\tau,N/G^1_\tau)\leq\aent_{\tau_q}(\phi^1_\tau)$. Then $\aent_\tau(\phi)\leq\aent_{\tau_q}(\phi^1_\tau)$, and hence $\aent_\tau(\phi)=\aent_{\tau_q}(\phi^1_\tau)$.
\end{proof}

\begin{remark}\label{wlog}
Let $(G,\tau)$ be a topological abelian group and let $\phi:(G,\tau)\to (G,\tau)$ be a continuous endomorphism. For the computation of the topological adjoint entropy of $\phi$ we can assume without loss of generality that $\tau$ is precompact and linear.

Indeed, by Proposition \ref{ent*=ent*1-tau} we can assume that $G$ is $\tau$-residually finite. Moreover, by Lemma \ref{res-fin-tau} this is equivalent to say that $\tau_\lambda^+$ is precompact.
Now, we can assume without loss of generality that $\tau=\tau_\lambda^+$ in view of Proposition \ref{lin+}, in other words, $\tau$ is precompact and linear.
\end{remark}

If in Proposition \ref{ent*=ent*1-tau} one considers a $\phi$-invariant subgroup of $G$ contained in $G^1_\tau$ instead of $G^1_\tau$ itself, then the same property holds. In particular, every connected subgroup $H$ of $G$ is contained in $G^1_\tau$, as each $N\in\CC_\tau(G)$ is clopen and so $N\supseteq H$.
So we have the following consequence of Proposition \ref{ent*=ent*1-tau}, showing the additivity of the topological adjoint entropy when one considers connected invariant subgroups.

\begin{corollary}\label{AT-conn}
Let $(G,\tau)$ be a topological abelian group, $\phi:(G,\tau)\to (G,\tau)$ a continuous endomorphism and $H$ a $\phi$-invariant subgroup of $(G,\tau)$. If $H$ is connected, then $\aent_{\tau\restriction_H}(\phi\restriction_H)=0$ and $\aent_\tau(\phi)=\aent_{\tau_q}(\overline\phi)$, where $\overline \phi:(G/H,\tau_q)\to (G/H,\tau_q)$ is the continuous endomorphism induced by $\phi$.
\end{corollary}

In particular, Corollary \ref{AT-conn} holds for the connected component of $(G,\tau)$, namely, $\aent_\tau(\phi)=\aent_{\tau_q}(\overline\phi)$, where $\overline\phi:((G,\tau)/c(G,\tau),\tau_q)\to ((G,\tau)/c(G,\tau),\tau_q)$ is the continuous endomorphism induced by $\phi$. Therefore, the computation of the topological adjoint entropy may be reduced to the case of continuous endomorphisms of totally disconnected abelian groups.

\section{Topological adjoint entropy, topological entropy and algebraic entropy}\label{aent-top-sec}

Since our aim is to compare the topological adjoint entropy with the topological entropy, we first recall the definition of topological entropy given by Adler, Konheim and McAndrew \cite{AKM}. 
For a compact topological space $X$ and for an open cover $\mathcal U$ of $X$, let $N(\mathcal U)$ be the minimal cardinality of a subcover of $\mathcal U$. Since $X$ is compact, $N(\mathcal U)$ is always finite. Let $H(\mathcal U)=\log N(\mathcal U)$ be the \emph{entropy of $\mathcal U$}.
For any two open covers $\mathcal U$ and $\mathcal V$ of $X$, let $\mathcal U\vee\mathcal V=\{U\cap V: U\in\mathcal U, V\in\mathcal V\}$. Define analogously $\mathcal U_1\vee\ldots	\vee\mathcal U_n$, for open covers  $\mathcal U_1,\ldots,\mathcal U_n$ of $X$.
Let $\psi:X\to X$ be a continuous map and $\mathcal U$ an open cover of $X$. Then $\psi^{-1}(\mathcal U)=\{\psi^{-1}(U):U\in\mathcal U\}$.
The \emph{topological entropy of $\psi$ with respect to $\mathcal U$} is $$H_{top}(\psi,\mathcal U)=\lim_{n\to\infty}\frac{H(\mathcal U\vee\psi^{-1}(\mathcal U)\vee\ldots\vee\psi^{-n+1}(\mathcal U))}{n},$$ and the \emph{topological entropy} of $\psi$ is $$h_{top}(\psi)=\sup\{H_{top}(\psi,\mathcal U):\mathcal U\ \text{open cover of $X$}\}.$$

\begin{remark}\label{BT}
Peters \cite{Pet} modified the above definition of algebraic entropy for automorphisms $\phi$ of arbitrary abelian groups $G$ using finite subsets instead of finite subgroups. In \cite{DG} this definition is extended in appropriate way to all endomorphisms of abelian groups, as follows.
For a non-empty finite subset $F$ of $G$ and for any positive integer $n$, the \emph{$n$-th $\phi$-trajectory} of $F$ is 
$T_n(\phi,F)=F+\phi(F)+\ldots+\phi^{n-1}(F).$
The limit $H(\phi,F)=\lim_{n\to\infty}\frac{\log|T_n(\phi,F)|}{n}$ exists and is the \emph{algebraic entropy of $\phi$ with respect to $F$}. The \emph{algebraic entropy} of $\phi$ is $$h(\phi)=\sup\{H(\phi,F):F\subseteq G\ \text{non-empty, finite}\}.$$
In particular, $\ent(\phi)=h(\phi\restriction_{t(G)})$.

In \cite{DG} we strengthen Theorem \ref{BT-W}, proving the following Bridge Theorem for the algebraic entropy $h$ in the general setting of endomorphisms of abelian groups:
\begin{quote}
\emph{Let $G$ be an abelian group and $\phi\in\End(G)$. Then $h(\phi)=h_{top}(\widehat\phi)$.}
\end{quote}
\end{remark}

We see now that it is in the proof of Weiss' Bridge Theorem \ref{BT-W} that one can find a first (non-explicit) appearance of the topological adjoint entropy. Indeed, that proof contains the following Lemma \ref{H*=Htop}, which gives a ``local equality'' between the topological adjoint entropy and the topological entropy.  In this sense it becomes natural to introduce the topological adjoint entropy (and the adjoint algebraic entropy as in \cite{DGS}).

\smallskip
Following \cite[Section 2]{W}, for a topological abelian group $(K,\tau)$ and $C\in\CC_\tau(G)$, let $$\zeta(C)=\{x + C: x\in G\},$$ which is an open cover of $K$. If $K$ is precompact, then $\zeta(C)$ is finite, since every open subgroup of $(K,\tau)$ has finite index in $K$. 

\begin{lemma}\label{H*=Htop}
Let $(K,\tau)$ be a compact abelian group and let $\psi:(K,\tau)\to (K,\tau)$ be a continuous endomorphism. If $C\in\CC_\tau(K)$, then $$H^\star(\psi,C)=H_{top}(\psi,\zeta(C)).$$
\end{lemma}
\begin{proof}
Let $C\in\CC_\tau(K)$. Obviously, $N(\zeta(C))= [K:C]$.
Moreover, one can prove by induction that $\zeta(\psi^{-n}(C))= \psi^{-n}(\zeta(C))$ for every $n\in\N_+$ and that $\zeta(C_1) \vee\ldots\vee\zeta(C_n)= \zeta(C_1\cap\ldots\cap C_n)$ for every $n\in\N_+$ and $C_1,\ldots,C_n\in\CC_\tau(K)$. This implies that, for every $n\in\N_+$, $$\zeta(C)\vee\psi^{-1}(\zeta(C))\vee\ldots\vee\psi^{-n+1}(\zeta(C))=\zeta(C\cap\psi^{-1}(C)\cap\ldots\cap\psi^{-n+1}(C))=\zeta(B_n(\psi,C)),$$ so that
\begin{align*}
N(\zeta(C)\vee\psi^{-1}\zeta(C)\vee\ldots\vee\psi^{-n+1}\zeta(C))&=N(\zeta(C\cap\psi^{-1}(C)\cap\ldots\cap\psi^{-n+1}(C)))\\
&=N(\zeta(B_n(\psi,C)))\\
&=[K:B_n(\psi,C)]\\
&=\card{C_n(\psi,C)}.
\end{align*}
Hence we have the thesis.
\end{proof}

In view of this lemma we can connect the topological adjoint entropy and the topological entropy for continuous endomorphisms of compact abelian groups:

\begin{theorem}\label{ent*=htop}
Let $(K,\tau)$ be a compact abelian group and let $\psi:(K,\tau)\to (K,\tau)$ be a continuous endomorphism. Then $$\aent_\tau(\psi)=\aent_{\tau_q}(\overline\psi)=h_{top}(\overline\psi),$$ where $\overline\psi:(K/c(K),\tau_q)\to (K/c(K),\tau_q)$ is the continuous endomorphism induced by $\psi$ and $\tau_q$ is the quotient topology induced by $\tau$ on $K/c(K)$.

In particular, if $(K,\tau)$ is totally disconnected, then $\ent_{\tau}^\star(\psi)=h_{top}(\psi)$.
\end{theorem}
\begin{proof}
Corollary \ref{AT-conn} gives $\aent_\tau(\psi)=\aent_{\tau_q}(\overline\psi)$.
Since $(K/c(K),\tau_q)$ is totally disconnected, every open cover of $K/c(K)$ is refined by an open cover of the form $\zeta(C)$, where $C\in\CC_{\tau_q}(K/c(K))$. Hence, 
\begin{equation}\label{w-eq}
h_{top}(\overline\psi)=\sup\{H_{top}(\overline\psi,\zeta(C)):C\in\CC_{\tau_q}(K/c(K))\}.
\end{equation}
By Lemma \ref{H*=Htop}, $\aent_{\tau_q}(\overline\psi)=h_{top}(\overline\psi)$.
\end{proof}

The equality in \eqref{w-eq} is contained also in the proof of \cite[Theorem 2.1]{W}. 

\medskip
As a corollary of Theorem \ref{ent*=htop} and of Weiss' Bridge Theorem \ref{BT-W} we obtain the following Bridge Theorem between the algebraic entropy and the topological adjoint entropy in the general case of endomorphisms of abelian groups; this is Theorem \ref{BT!-intro} of the introduction.

\begin{corollary}[Bridge Theorem]\label{BT!}
Let $(K,\tau)$ be a compact abelian group and $\psi:(K,\tau)\to(K,\tau)$ a continuous endomorphism. Then $$\aent_\tau(\psi)=\ent(\widehat\psi).$$
\end{corollary}
\begin{proof}
Let $G=\widehat K$ and $\phi=\widehat\psi$.
By the definition of algebraic entropy and by Theorem \ref{ent*=htop} we have respectively
\begin{equation}\label{uno}
\ent(\phi)=\ent(\phi\restriction_{t(G)}) \ \text{and}\ \aent_{\widehat\tau}(\psi)=\aent_{\widehat\tau_q}(\overline\psi),
\end{equation}
 where $\overline\psi:(K/c(K),\widehat\tau_q)\to (K/c(K),\widehat\tau_q)$ is the continuous endomorphism induced by $\psi$ and $\widehat\tau_q$ is the quotient topology of $\widehat\tau$. By the Pontryagin duality $\widehat{t(G)}$ is topologically isomorphic to $K/c(K)$ and $\widehat{\phi\restriction_{t(G)}}$ is conjugated to $\overline\psi$ by a topological isomorphism. The topological entropy is invariant under conjugation by topological isomorphisms, so 
\begin{equation}\label{due}
h_{top}(\widehat{\phi\restriction_{t(G)}})=h_{top}(\overline\psi).
\end{equation} 
Now Theorem \ref{BT-W} and Theorem \ref{ent*=htop} give respectively
\begin{equation}\label{tre}
\ent(\phi\restriction_{t(G)})=h_{top}(\widehat{\phi\restriction_{t(G)}})\ \text{and}\ \aent_{\widehat\tau_q}(\overline\psi)=h_{top}(\overline\psi).
\end{equation}
The thesis follows from \eqref{uno}, \eqref{due} and \eqref{tre}.
\end{proof}

This Bridge Theorem shows that the topological adjoint entropy $\aent_\tau$ is the topological counterpart of the algebraic entropy $\ent$, as well as the topological entropy $h_{top}$ is the topological counterpart of the algebraic entropy $h$ in view of the Bridge Theorem stated in Remark \ref{BT}. Indeed, roughly speaking, $\aent_\tau$ reduces to totally disconnected compact abelian groups, as dually the algebraic entropy $\ent$ reduces to torsion abelian groups.

\medskip
The following is another Bridge Theorem, showing that the topological adjoint entropy of a continuous endomorphism of a precompact abelian group is the same as the algebraic entropy of the dual endomorphism in the precompact duality. Its Corollary \ref{cor} will show that this entropy coincides also with the algebraic entropy of the dual endomorphism in the Pontryagin duality of its extension to the completion.

\begin{theorem}[Bridge Theorem]\label{ent-adj-top}
Let $(G,\tau)$ be a precompact abelian group and $\phi:(G,\tau)\to(G,\tau)$ a continuous endomorphism. Then $$\aent_\tau(\phi)=\ent(\prdual\phi).$$
\end{theorem}
\begin{proof} 
Let $N\in\mathcal C_\tau(G)$. Then $F=N^\bot$ is a finite subgroup of $\prdual{(G,\tau)}$ by Fact \ref{prd}(a). By Fact \ref{prd}(c) $(\phi^{-n}(N))^\bot=(\prdual\phi)^n (F)$ for every $n\in\N$. Hence, $B_n(\phi,N)^\bot=T_n(\prdual \phi,F)$ for every $n\in\N_+$ by Fact \ref{prd}(b). It follows that $$|C_n(\phi,N)|=|\prdual{C_n(\phi,N)}|=|B_n(\phi,N)^\bot|=|T_n(\prdual\phi,F)|$$ for every $n\in\N_+$, and this concludes the proof.
\end{proof}

\begin{corollary}\label{cor}
Let $(G,\tau)$ be a precompact abelian group. Denote by $\widetilde\tau$ the topology of the completion $K=\widetilde{(G,\tau)}$, let $\widetilde\phi:K\to K$ be the extension of $\phi$, and let $\overline{\widetilde\phi}:K/c(K)\to K/c(K)$ be the endomorphism induced by $\widetilde\phi$. Then 
$$\ent(\prdual\phi)=\aent_\tau(\phi)=\aent_{\widetilde\tau}(\widetilde\phi)=\ent(\widehat{\widetilde\phi})=h_{top}(\overline{\widetilde\phi}).$$
\end{corollary}
\begin{proof}
The first equality is Theorem \ref{ent-adj-top}, the second is Corollary \ref{aenttilde}, while the equality $\aent_{\widetilde\tau}(\widetilde\phi)=h_{top}(\overline{\widetilde\phi})$ follows from Theorem \ref{ent*=htop} and $\ent(\widehat{\widetilde\phi})=h_{top}(\widetilde\phi)$ from Theorem \ref{BT-W}.
\end{proof}

In the hypotheses of Corollary \ref{cor}, assume that $(G,\tau)$ is linear. Then $\widetilde{(G,\tau)}$ is linear as well and we obtain
$$\ent(\prdual\phi)=\aent_\tau(\phi)=\aent_{\widetilde\tau}(\widetilde\phi)=\ent(\widehat{\widetilde\phi})=h_{top}(\widetilde\phi).$$
If one assumes only that $(G,\tau)$ is a totally disconnected compact abelian group, then $\widetilde{(G,\tau)}$ is not totally disconnected in general (it can be connected), and the equality $\aent_{\widetilde\tau}(\widetilde\phi)=h_{top}(\widetilde\phi)$ may fail as the following example shows.

\begin{example}\label{exx}
Let $G=\Z$ and denote by $\widetilde\tau$ the compact topology of $\T$. Let $\alpha\in\T$ be an element of infinite order, consider the inclusion $\Z\to \T$ given by $1\mapsto \alpha$ and endow $\Z$ with the topology $\tau_\alpha$ inherited from $\T$ by this inclusion. Then $(G,\tau_\alpha)$ is dense in $\T$. Moreover, consider $\widetilde\phi=\mu_2:\T\to \T$, defined by $\widetilde\phi(x)=2x$ for every $x\in \T$, and $\phi=\mu_2\restriction_G$.

Now $\aent_{\tau_\alpha}(\phi)=0$, by Proposition \ref{ent-adj-top}(b) we have $\aent_{\tau}(\widetilde\phi)=\aent_{\tau_\alpha}(\phi)$, and so $\aent_{\widetilde\tau}(\widetilde\phi)=0$. 

On the other hand, the Kolmogorov-Sinai Theorem for the topological entropy of automorphisms of $\T^n$ (see \cite{K,S}) gives $h_{top}(\widetilde\phi)=\log 2$.
\end{example}

The situation described by Example \ref{exx} can be generalized as follows. Suppose that $(G,\tau)$ is a precompact abelian group such that $G$ is $\tau$-residually finite. Let $\phi:(G,\tau)\to(G,\tau)$ be a continuous endomorphism and $\widetilde\phi:\widetilde{(G,\tau)}\to \widetilde{(G,\tau)}$ the extension of $\phi$ to $\widetilde{(G,\tau)}$. Since $G$ is $\tau$-residually finite, $\tau_\lambda=\tau_\lambda^+$ is precompact by Lemma \ref{res-fin-tau}. Denote by $\phi_\lambda$ the endomorphism $\phi$ considered on $(G,\tau_\lambda)$. Then $\phi_\lambda:(G,\tau_\lambda)\to (G,\tau_\lambda)$ is continuous, and let $\widetilde\phi_\lambda:\widetilde{(G,\tau_\lambda)}\to \widetilde{(G,\tau_\lambda)}$ be the extension of $\phi$ to $\widetilde{(G,\tau_\lambda})$. Denote by $\widetilde\tau$ and $\widetilde\tau_\lambda$ respectively the topology of $\widetilde{(G,\tau)}$ and $\widetilde{(G,\tau_\lambda)}$.
The situation is described by the following diagram.
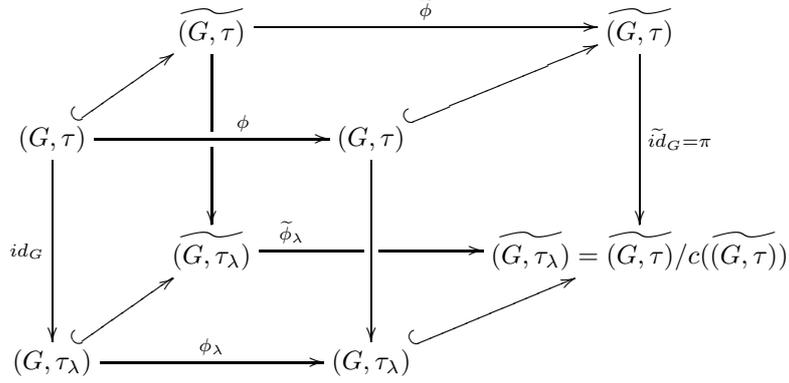
\begin{equation*}
\xymatrix{
& \widetilde{(G,\tau)} \ar[rr]^{\widetilde\phi}\ar'[d][dd] & & \widetilde{(G,\tau)} \ar[dd]^{\widetilde{id}_G=\pi}\\
(G,\tau)\ar[rr]^{\ \ \ \ \ \ \ \phi}\ar@{^{(}->}[ru] \ar[dd]_{id_G}& & (G,\tau) \ar@{^{(}->}[ru] \ar[dd]&\\
& \widetilde{(G,\tau_\lambda)}\ar'[r]^{\widetilde\phi_\lambda}[rr] & & \widetilde{(G,\tau_\lambda)}=\widetilde{(G,\tau)}/c(\widetilde{(G,\tau)})\\
(G,\tau_\lambda)\ar[rr]^{\phi_\lambda}\ar@{^{(}->}[ru] & & (G,\tau_\lambda) \ar@{^{(}->}[ru]&\\
}
\end{equation*}
The right side face of the cubic diagram (which obviously coincides with the left one) is described by \eqref{leftside}, that is, $\widetilde{(G,\tau_\lambda)}=\widetilde{(G,\tau)}/c(\widetilde{(G,\tau)})$ and the extension $\widetilde{id}_G$ of $id_G:(G,\tau)\to (G,\tau_\lambda)$ coincides with the canonical projection $\pi:\widetilde{(G,\tau)}\to \widetilde{(G,\tau)}/c(\widetilde{(G,\tau)})$.

So in the cubic diagram we have four continuous endomorphisms, namely $\phi$, $\phi_\lambda$, $\widetilde\phi$ and $\widetilde\phi_\lambda$, that have the same topological adjoint entropy; indeed, Proposition \ref{lin+} covers the front face of the cube (i.e., $\aent_\tau(\phi)=\aent_{\tau_\lambda}(\phi_\lambda)$), while Corollary \ref{aenttilde} applies to the upper and the lower faces (i.e., $\aent_\tau(\phi)=\aent_{\widetilde\tau}(\widetilde\phi)$ and $\aent_{\tau_\lambda}(\phi_\lambda)=\aent_{\widetilde\tau_\lambda}(\widetilde\phi_\lambda)$). 

On the other hand we can consider the topological entropy only for $\widetilde\phi$ and $\widetilde\phi_\lambda$, as the completions are compact. By Theorem \ref{ent*=htop}, $h_{top}(\widetilde\phi_\lambda)=\ent_{\widetilde\tau_\lambda}(\phi_\lambda)$ and so $h_{top}(\widetilde\phi_\lambda)$ coincides with the topological adjoint entropy of the four continuous endomorphisms. For the topological entropy of $\widetilde\phi$ we can say only that $h_{top}(\widetilde\phi)\geq h_{top}(\widetilde\phi_\lambda)$, in view of the monotonicity property of the topological entropy.

We collect all these observations in the following formula:
\begin{equation*}\label{cube}
h_{top}(\widetilde\phi)\geq h_{top}(\widetilde\phi_\lambda)=\aent_{\widetilde\tau_\lambda}(\widetilde\phi_\lambda)=\aent_{\tau_\lambda}(\phi_\lambda)=\aent_\tau(\phi)=\aent_{\widetilde\tau}(\widetilde\phi).
\end{equation*}
The next example shows that this inequality can be strict in the strongest way; indeed, we find a $\phi$ such that $h_{top}(\widetilde\phi)=\infty$ and $h_{top}(\widetilde\phi_\lambda)=0$.

\begin{example}
Let $G=\bigoplus_{p\in\Prm}\Z(p)$. Since $G$ is residually finite, $\gamma_G$ is Hausdorff by Lemma \ref{res-fin-tau}, and so $\mathcal P_G$ is Hausdorff as well, as $\gamma_G\leq\mathcal P_G$ (see \eqref{pnb}). Moreover, $\gamma_G=\nu_G$.

Let $\mu_2$ denote the endomorphism given by the multiplication by $2$, and set $\phi=\mu_2:G\to G$.

Consider $\phi:(G,\mathcal P_G)\to(G,\mathcal P_G)$. Then $\phi$ can be extended to the Bohr compactification $K=bG$ of $G$ ($K$ coincides with the completion of $(G,\mathcal P_G)$) and let $\widetilde\phi:K\to K$ be this extension. By the density of $G$ in $K$, and the uniqueness of the extension, $\widetilde\phi=\mu_2:K\to K$. By \cite[Theorem 3.5 - Case 1 of the proof]{ADS} with $\alpha=\mu_2$, we have $h_{top}(\widetilde\phi)=\infty$. 

Consider now $\phi_\lambda:(G,\gamma_G)\to (G,\gamma_G)$, where $(\mathcal P_G)_\lambda=\gamma_G$ by Lemma \ref{bl=p}.
The completion $K_\lambda$ of $(G,\gamma_G)$ is $K_\lambda=\prod_{p\in\Prm}\Z(p)$, as $\gamma_G=\nu_G$, and $\phi$ extends to $\widetilde \phi_\lambda:K_\lambda\to K_\lambda$; denote by $\widetilde\gamma_G$ the topology of $K_\lambda$. Moreover, $\aent_{\widetilde\gamma_G}(\widetilde\phi_\lambda)=0$, because every finite-index subgroup of $K_\lambda$ is $\widetilde\phi_\lambda$-invariant. By Theorem \ref{ent*=htop} we have $h_{top}(\widetilde\phi_\lambda)=\aent_{\widetilde\gamma_G}(\widetilde\phi_\lambda)$ and so $h_{top}(\widetilde\phi_\lambda)=0$.
\end{example}

\section{The topological adjoint entropy of the Bernoulli shifts}\label{bernoulli}

In the following example we recall the proof given in \cite{DGS} of the fact that the adjoint algebraic entropy of the Bernoulli shifts is infinite.

\begin{example}\label{beta}
If $K$ is a non-trivial abelian group, then $$\aent(\beta^\oplus_K)=\aent({}_K\beta^\oplus)=\aent(\overline\beta_K^\oplus)=\infty.$$
Indeed, by \cite[Corollary 6.5]{AGB}, $\ent(\beta_K)=\ent({}_K\beta)=\ent(\overline\beta_K)=\infty$. Moreover, $\widehat{\beta_K^\oplus}={}_K\beta$,
$\widehat{{}_K\beta^\oplus}=\beta_K$, and $\widehat{\overline\beta_K^\oplus}=(\overline\beta_K)^{-1}$ by \cite[Proposition 6.1]{DGS}.
Then Theorem \ref{BT-DGS} implies
$\ent^\star(\beta_K^\oplus)=\ent({}_K\beta)=\infty$, $\ent^\star({}_K\beta^\oplus)=\ent(\beta_K)=\infty$ and $\ent^\star(\overline\beta_K^\oplus)=\ent((\overline\beta_K)^{-1})=\infty$.
\end{example}

We give now a direct computation of the value of the algebraic adjoint entropy of the Bernoulli shifts. The starting idea for this proof was given to me by Brendan Goldsmith and Ketao Gong.

\begin{proposition}\label{beta-d}
Let $p$ be a prime and $K=\Z(p)$. Then $$\aent(\beta^\oplus_K)=\aent({}_K\beta^\oplus)=\aent(\overline\beta_K^\oplus)=\infty.$$ 
\end{proposition}
\begin{proof}
Since $K^{(\N)}\cong K^{(\N_+)}$, we consider without loss of generality $K^{(\N_+)}$ instead of $K^{(\N)}$, since it is convenient for our proof.
Write $K^{(\N)}=\bigoplus_{n\in\N_+}\hull{e_n}$, where $\{e_n:n\in\N_+\}$ is the canonical base of $K^{(\N)}$.

\smallskip
(i) We start proving that $\aent(\beta^\oplus_K)=\infty$.
Define, for $i\in\N$,
\begin{align*}
\delta_1(i)&=\begin{cases}1 & \text{if}\ i=(2n)!+n\ \text{for some}\ n\in\N_+, \\ 0 & \text{otherwise}.\end{cases}\\
\delta_2(i)&=\begin{cases}1 & \text{if}\ i=(2n+1)!+n\ \text{for some}\ n\in\N_+, \\ 0 & \text{otherwise}.\end{cases}
\end{align*}
Let $N_2=\hull{e_i-\delta_1(i)e_1-\delta_2(i)e_2:i\geq 3}\in\CC(K^{(\N)})$. Then $e_1,e_2\not\in N_2$ and $K^{(\N)}=N_2\oplus\hull{e_1}\oplus\hull{e_2}$. For every $n\in\N_+$, 
\begin{equation}\label{b1}
\hull{e_{(2n)!},e_{(2n+1)!}}\subseteq B_{n}(\beta_K^\oplus,N_2);
\end{equation}
indeed, $(\beta_K^\oplus)^k(e_{(2n)!})=e_{(2n)!+k}\in N_2$ and $(\beta_K^\oplus)^k(e_{(2n+1)!})=e_{(2n+1)!+k}\in N_2$ for every $k\in\N$ with $k<n$.
Moreover, 
\begin{equation}\label{b2}
\hull{e_{(2n)!},e_{(2n+1)!}}\cap B_{n+1}(\beta_K^\oplus,N_2)=0.
\end{equation}
In fact, since $\hull{e_{(2n)!},e_{(2n+1)!}}\subseteq B_{n}(\beta_K^\oplus,N_2)$, we have $a_1e_{(2n!)}+a_2e_{(2n+1)!}\in B_{n+1}(\beta_K^\oplus,N_2)$ for $a_1,a_2\in\Z(p)$ if and only if $(\beta_K^\oplus)^n(a_1e_{(2n!)}+a_2e_{(2n+1)!})=a_1e_{(2n!+n)}+a_2e_{(2n+1)!+n}\in N_2$; since $e_{(2n!)+n}-e_1\in N_2$ and $e_{(2n+1)!+n}-e_2\in N_2$, this is equivalent to $a_1e_1+a_2e_2\in N_2$, which occurs if and only if $a_1=a_2=0$.

By \eqref{b1} and \eqref{b2}, for every $n\in\N_+$, $$\left|\frac{B_n(\beta_K^\oplus,N_2)}{B_{n+1}(\beta_K^\oplus,N_2)}\right|\geq p^2.$$ So \eqref{C=C_n->ent*=0} gives $H(\beta_K^\oplus,N_2)\geq 2\log p$.

Generalizing this argument, for $m\in\N_+$, $m>1$, define, for $i\in\N$,
\begin{align*}
\delta_1(i)&=\begin{cases}1 & \text{if}\ i=(mn)!+n\ \text{for some}\ n\in\N_+, \\ 0 & \text{otherwise}.\end{cases}\\
\delta_2(i)&=\begin{cases}1 & \text{if}\ i=(mn+1)!+n\ \text{for some}\ n\in\N_+, \\ 0 & \text{otherwise}.\end{cases}\\
&\vdots\\
\delta_m(i)&=\begin{cases}1 & \text{if}\ i=(mn+m-1)!+n\ \text{for some}\ n\in\N_+, \\ 0 & \text{otherwise}.\end{cases}
\end{align*}
Let $N_m=\hull{e_i-\delta_1(i)e_1-\ldots-\delta_m(i)e_m:i\geq m+1}\in\CC(K^{(\N)})$. Then $e_1,\ldots,e_m\not\in N_m$ and $K^{(\N)}=N_m\oplus\hull{e_1}\oplus\ldots\oplus\hull{e_m}$. 
For every $n\in\N_+$, 
\begin{equation}\label{bm1}
\hull{e_{(mn)!},\ldots,e_{(mn+m-1)!}}\subseteq B_{n}(\beta_K^\oplus,N_m);
\end{equation}
indeed, $(\beta_K^\oplus)^k(e_{(mn)!})=e_{(mn)!+k}\in N_m,\ldots,(\beta_K^\oplus)^k(e_{(mn+m-1)!})=e_{(mn+1)!+m-1}\in N_m$ for every $k\in\N$ with $k<n$.
 
Moreover,
\begin{equation}\label{bm2}
\hull{e_{(mn)!},e_{(mn+1)!}}\cap B_{n+1}(\beta_K^\oplus,N_m)=0.
\end{equation}
In fact, since $\hull{e_{(mn)!},\ldots,e_{(mn+m-1)!}}\subseteq B_{n}(\beta_K^\oplus,N_m)$, we have $a_1e_{(mn!)}+\ldots+a_me_{(mn+m-1)!}\in B_{n+1}(\beta_K^\oplus,N_m)$ for $a_1,\ldots,a_m\in\Z(p)$ if and only if $(\beta_K^\oplus)^n(a_1e_{(mn!)}+\ldots+a_me_{(mn+m-1)!})=a_1e_{(mn!+n)}+\ldots+a_me_{(mn+m-1)!+n}\in N_m$; since $e_{(mn!)+n}-e_1\in N_m,\ldots,e_{(mn+m-1)!+n}-e_{m}\in N_m$, this is equivalent to $a_1e_1+\ldots+a_me_m\in N_m$, which occurs if and only if $a_1=\ldots=a_m=0$.

By \eqref{bm1} and \eqref{bm2}, for every $n\in\N_+$, $$\left|\frac{B_n(\beta_K^\oplus,N_m)}{B_{n+1}(\beta_K^\oplus,N_m)}\right|\geq p^m.$$  So \eqref{C=C_n->ent*=0} yields $H(\beta_K^\oplus,N_m)\geq m\log p$.

We have seen that $\aent(\beta_K^\oplus)\geq m\log p$ for every $m\in\N_+$ and so $\aent(\beta_K^\oplus)=\infty$.

\smallskip
(ii) The same argument as in (i) shows that $\aent(\overline\beta_K^\oplus)=\infty$.

\smallskip
(iii) An analogous argument shows that $\aent({}_K\beta^\oplus)=\infty$.
Indeed, it suffices to define, for every $m\in\N_+$, $m>1$,
\begin{align*}
\delta_1'(i)&=\begin{cases}1 & \text{if}\ i=(mn)!-n\ \text{for some}\ n\in\N_+, \\ 0 & \text{otherwise}.\end{cases}\\
\delta_2'(i)&=\begin{cases}1 & \text{if}\ i=(mn+1)!-n\ \text{for some}\ n\in\N_+, \\ 0 & \text{otherwise}.\end{cases}\\
&\vdots\\
\delta_m'(i)&=\begin{cases}1 & \text{if}\ i=(mn+m-1)!-n\ \text{for some}\ n\in\N_+, \\ 0 & \text{otherwise},\end{cases}
\end{align*}
$N_m=\hull{e_i-\delta_1'(i)e_1-\ldots-\delta_m'(i)e_m:i\geq m+1}\in\CC(K^{(\N)})$, and proceed as in (i).
\end{proof}

\begin{example}
Let $K$ be a non-trivial abelian group.

\smallskip
Example \ref{beta} (and so Proposition \ref{beta-d}) is equivalent to
$\aent_{\gamma_{K^{(\N)}}}(\beta_K^\oplus)=\aent_{\gamma_{K^{(\N)}}}({}_K\beta^\oplus)=\aent_{\gamma_{K^{(\Z)}}}(\overline\beta_K)=\infty$.

\smallskip
On the other hand, we can consider the topology $\tau_K$ on $K^{(\N)}$ (respectively, $K^{(\Z)}$) induced by the product topology of $K^{\N}$ (respectively, $K^{\Z}$). Clearly, $\tau_K\leq \gamma_{K^{(\N)}}$ (respectively, $\tau_K\leq \gamma_{K^{(\Z)}}$). Moreover, 
\begin{center}
$\CC_{\tau_K}(K^{(\N)})=\{0^F\oplus K^{(\N\setminus F)}:F\subseteq\N, F\text{ finite}\}$
\end{center}
\begin{center}
(respectively, $\CC_{\tau_K}(K^{(\Z)})=\{0^F\oplus K^{(\Z\setminus F)}:F\subseteq\Z, F\text{ finite}\}$).
\end{center}

Then we see that 
\begin{itemize}
\item[(i)]$\aent_{\tau_K}(\beta_K^\oplus)=0$ and 
\item[(ii)]$\aent_{\tau_K}({}_K\beta^\oplus)=\aent_{\tau_K}(\overline\beta_K^\oplus)=\log|K|$.
\end{itemize}
Indeed, every $N\in\CC_{\tau_K}(K^{(\N)})$ contains $$N_m=\underbrace{0\oplus\ldots\oplus0}_m\oplus K^{\N\setminus\{0,\ldots,m-1\}}$$ for some $m\in\N_+$. By Lemma \ref{N<M->HN>HM} it suffices to calculate the topological adjoint entropy with respect to the $N_m$.

\smallskip
(i) Let $m\in\N_+$. For every $n\in\N_+$, $B_n(\beta_K^\oplus,N_m)=N_m$ and so $H(\beta_K^\oplus,N_m)=0$. Therefore, $\aent_{\tau_K}(\beta_K^\oplus)=0$. 

\smallskip
(ii) Let $m\in\N_+$. For every $n\in\N_+$, $B_n({}_K\beta^\oplus,N_m)=N_{m+n-1}$. So $|C_n({}_K\beta^\oplus,N_m)|=|K|^{m+n-1}$ and hence $H^\star({}_K\beta^\oplus,N_m)=\log|K|$. Consequently, $\aent_{\tau_K}({}_K\beta^\oplus)=\log|K|$. The proof that $\aent_{\tau_K}(\overline\beta_K^\oplus)=\log|K|$ is analogous.
\end{example}

\end{document}